\documentclass[10pt]{amsart}
\usepackage{graphicx}
\usepackage[bf]{caption}
\usepackage{epstopdf,epsfig}
\usepackage[hidelinks]{hyperref}

\usepackage{amsmath,amsthm,amsfonts,amssymb,latexsym,amscd,enumerate}

\usepackage{palatino}
\usepackage[mathcal]{euler}

\usepackage{xcolor}

\usepackage{xy}
\xyoption{all}

\swapnumbers

\theoremstyle{plain}
\newtheorem*{theorem*}{Theorem} 
\newtheorem*{lemma*}{Lemma}
\newtheorem*{assumption*}{Assumption}

\newtheorem{theorem}[equation]{Theorem} 
\newtheorem{lemma}[equation]{Lemma}

\newtheorem{proposition}[equation]{Proposition}

\theoremstyle{definition}
\newtheorem{definition}[equation]{Definition}
\newtheorem{assumption}[equation]{Assumption}
\newtheorem{example}[equation]{Example}

\newtheorem*{remark*}{Remark}

\theoremstyle{remark}

\numberwithin{equation}{section}

\newcommand{\R}{\mathbb{R}}
\newcommand{\C}{\mathbb{C}}

\newcommand{\HH}{\mathfrak{H}}

\newcommand{\Wr}{\operatorname{Wr}}

\newcommand{\Spec}{\operatorname{Spec}}
\newcommand{\dom}{\operatorname{dom}}

\DeclareMathOperator{\Trace}{Trace}

\begin{document}

\title{On a Spectral Theorem of Weyl}

\author{Nigel Higson and Qijun Tan 
}


\date{}

\maketitle

\begin{abstract}
\noindent  We give a new proof of a theorem of Weyl on the continuous part of the spectrum of Sturm-Liouville operators on the half-line with asymptotically constant coefficients.  Earlier arguments, due to Weyl and Kodaira, depended on particular features of Green's functions for linear ordinary differential operators. Ours uses a  concept of asymptotic containment of $C^*$-algebra representations that has geometric origins. We  apply the concept elsewhere to the Plancherel formula for spherical functions on   reductive groups.
\end{abstract}


\section{Introduction}
\label{sec-introduction}
The purpose of this paper is to present a new approach to an old theorem of Hermann Weyl on the spectral theory of self-adjoint Sturm-Liouville operators on a half-line.  Our aim is to introduce  methods that are more geometric  and more amenable to generalization than the originals. We show elsewhere \cite{Tan19} that the same arguments apply to Harish-Chandra's Plancherel formula for spherical functions (in fact Harish-Chandra was very much inspired by Weyl's work; compare \cite[p.\;38]{Borel01} and \cite{vandenBan08}).

Sturm-Liouville theory is of course concerned with the eigenvalues and eigenfunctions of linear differential operators
\begin{equation}
  \label{eq-SL-operator}
  D = - \frac{d\,\,}{dx} \cdot p(x) \cdot \frac{d\,\,} {dx} + q(x) ,
\end{equation}
initially on a closed interval $[a,b]$.  Assume for simplicity that $p(x)$ and $q(x)$ are smooth, real-valued functions on $[a,b]$, with $p(x)$ positive everywhere.  In examining the nonzero solutions of the eigenvalue problem
\begin{equation}
  \label{eq-eigenvalue-eqn}
  D f _ \lambda = \lambda f_\lambda ,
\end{equation}
it is appropriate to impose suitable self-adjoint boundary conditions.  For the sake of this introduction let us choose the simplest of these, namely
\begin{equation}
  \label{eq-eigenvalue-bdy-conds}
  f_\lambda(a) = 0 = f_\lambda (b).
\end{equation}
The elements of Sturm-Liouville theory can then be summarized as follows:

\begin{theorem}
\label{thm-classical-sturm-liouville}
  The eigenvalues $\lambda$ for the above problem are real numbers, and each has multiplicity one.  The set of all eigenvalues is a discrete subset of $\R$, bounded below, and if $h$ is any smooth function on $[a,b]$, then
  \[
    h(x) = \sum_{\lambda } \frac{\langle f_\lambda , h\rangle }{\langle f_\lambda, f_\lambda\rangle } f_\lambda (x)
  \]
  for $x\in (a,b)$. \end{theorem}

In an influential paper from early in his career, Weyl developed an analogous theory for Sturm-Liouville operators on a half-line rather than a bounded interval   \cite{Weyl10}. Weyl's paper addresses many issues, but our concern here is his treatment of the continuous spectrum of certain  Sturm-Liouville operators, and especially his version, for the continuous spectrum, of the expansion theorem above.

Assume that the coefficient functions $p(x)$ and $q(x)$ in \eqref{eq-SL-operator} are defined on $[0,\infty)$ and converge sufficiently rapidly to the constants $1$ and $0$, respectively, as $x$ tends to infinity.  For the purpose of this introduction, let us assume more than Weyl, namely that
\begin{equation}
  \label{eq-simple-asymptotic-condition}
  p(x) \equiv 1 \quad\text{and} \quad  q(x) \equiv 0 \qquad \text{if $x \gg 0$}
\end{equation}
(this assumption is too strong to be interesting in applications, but it allows us to quickly introduce Weyl's result). For each $\lambda\in \C$ there is a one-dimensional space of eigenfunctions
$F_\lambda$ for $D$ that satisfy the boundary condition
\begin{equation}
  \label{eq-SL-boundary-cond2}
  F_\lambda (0) =0.
\end{equation}
If we focus on the case where $\lambda >0$, and if we choose, as we may, $F_\lambda$ to be nonzero and real-valued, then our assumptions on the coefficient functions $p(x)$ and $q(x)$ imply that
\begin{equation}
  \label{eq-c-fn-def1}
  F_\lambda (x) = c(\lambda)e^{i\sqrt{\lambda}\, x} + \overline{c(\lambda}) e ^{-i\sqrt{\lambda}\, x} 
\end{equation}
for some nonzero $c(\lambda)\in \C$ and all $x\gg 0$.  Weyl's result for the continuous spectrum, expressed in $L^2$-terms, is as follows (there is also a pointwise result that is analogous to Theorem~\ref{thm-classical-sturm-liouville}, which may be   derived from the $L^2$-result):

\begin{theorem}
  \label{thm-weyl}
  If $h$ and $g$ are   smooth and compactly supported functions on $(0,\infty)$, then
 \[
 \langle f,g\rangle   = \sum_{\lambda < 0} \frac{\langle g, F_\lambda \rangle\langle F_\lambda , h\rangle}{\langle F_\lambda, F_\lambda\rangle} + \frac{1}{4 \pi} \int _0^\infty \frac{\langle g, F_\lambda \rangle\langle F_\lambda , h\rangle}{\,\, | c(\lambda) |^2}\, \frac{d\lambda} {\sqrt{\lambda}} ,
  \]
  The  sum is over the \emph{square-integrable} eigenfunctions associated to negative eigenvalues that satisfy the boundary condition \eqref{eq-SL-boundary-cond2}, and there are finitely many of these.  The integral   is absolutely convergent.  All the inner products in the formula have the standard $L^2$-form.
\end{theorem}

We shall approach Weyl's theorem by comparing the Sturm-Liouville operator $D$ to the simpler operator
\[
  D_0 = - \frac{d^2 \phantom{i}}{dx^2} .
\]
But we shall regard $D_0$ as an operator on the full line $(-\infty,\infty)$, rather than the half-line, and the translation-invariance of $D_0$ on the full line will be crucial.  To explain why,  it is helpful to make the following general definition:

\begin{definition}
  \label{def-asymptotic-containment}
  Let $A$ be a $C^*$-algebra and let
  \[
    \pi\colon A \longrightarrow B(H) \quad \text{and} \quad \pi_0 \colon A \longrightarrow B(H_0)
  \]
  be nondegenerate Hilbert space representations of $A$.  We shall say that $\pi_0$ is \emph{asymptotically contained} in $\pi$ if
  \begin{enumerate}[\rm (i)]
    \item There is a one-parameter group  of unitary operators   $U_t\colon H_0{\to} H_0$ that commute with the operators in $\pi_0[A]$. 
    \item There is a bounded operator $ W \colon H_{0} {\to} H $ such that for every $a \in A$, and for every $u,v \in H_0$,
    \begin{equation}
      \label{eq-asymptotic relation-intro}
      \lim_{t\to +\infty} 
      \left  [ 
        \bigl\langle   u, \pi_0(a)   v \bigr \rangle_{H_0} 
        -
        \bigl\langle  W U_t u,    \pi (a)W U_t v  \bigr \rangle _H  
      \right  ]  = 0  .
    \end{equation}
  \end{enumerate}
\end{definition}
Note that asymptotic containment of representations implies weak containment of representations \cite[Definition 3.4.5]{Dixmier77}.

For Weyl's theorem, we  take $A {=} C_0(\mathbb{R})$, and we  define $\pi$ and $\pi_0$ to be the functional calculus representations
\[
  \pi \colon \varphi \longmapsto \varphi (D) \quad \text{and} \quad \pi_0\colon \varphi \longmapsto \varphi (D_0),
\]
on $H{=}L^2 (0,\infty)$ and $H_0{=}L^2 (-\infty, \infty)$ respectively. We define $\{ U_t\}$ to be the standard one-parameter unitary group of translations on $L^2 (-\infty , \infty)$ and take
\[
  W \colon L^2 (-\infty, \infty) \longrightarrow L^2(0,\infty)
\]
to be the obvious restriction operator. The asymptotic containment of $\pi_0$ in $\pi$ follows from the condition \eqref{eq-simple-asymptotic-condition} on the coefficients of $D$; see Lemma~\ref{lem-asymp-inclusion}.

Returning to the general case,   we shall assume that the  $C^*$-algebra $A$ is separable and \emph{commutative}, as it is in the examples of concern to us, and that the Hilbert spaces $H$ and $H_0$ are separable, too.  Then the  representations $\pi$ and $\pi_0$ may be decomposed into direct integrals
\begin{equation}
  \label{eq-two-direct-integrals}
  H = \int^\oplus H_\lambda \, d\mu(\lambda)
  \quad\text{and} \quad
  H_0 = \int ^\oplus H_{0,\lambda} \, d \mu_0(\lambda)  
\end{equation}
over the spectrum of $A$, so that the action of $a{\in} A$ on each space in either integral is through the character $a \mapsto \lambda(a)  $.  We shall assume that the spaces $H_\lambda $ and $H_{0,\lambda}$ are \emph{finite-dimensional}, as again is the case in the examples of concern to us.

Now  let us assume temporarily that the asymptotic containment relation \eqref{eq-asymptotic relation-intro} is replaced by the \emph{exact} containment relation
\begin{equation}
  \label{eq-exact-relation}
  \bigl\langle   u, \pi_0(a)   v \bigr \rangle_{H_0} 
  =
  \bigl\langle  W   u,    \pi (a)W   v  \bigr \rangle _H  
\end{equation}
for all $u, v \in H_0$ and all $a\in A$, so that the operator $W$ is necessarily an isometric inclusion of $\pi_0$ as a subrepresentation of $\pi$. It follows  that the operator $W$ decomposes into a field of operators
\begin{equation}
  \label{eq-w-lambda}
  W_\lambda \colon H_{0,\lambda } \longrightarrow H_{\lambda} .
\end{equation}
and that each $W_\lambda^* W_\lambda$ is a multiple of the identity operator on $H_{0,\lambda}$. Of course that multiple is $  {\Trace ( W^*_\lambda W_\lambda ) }  /{\dim (H_{0,\lambda})}$.

Using the family $\{ W_\lambda \}$,   the   direct integral decomposition of $\pi$  in \eqref{eq-two-direct-integrals} gives rise to  an alternative direct integral decomposition of $\pi_0$.  Comparing the two, we find   that the measure $ \mu_0$ in \eqref{eq-two-direct-integrals} is absolutely continuous with respect to $\mu$,  and that indeed \begin{equation}
  \label{eq-radon-nikodym-fmla}
  \frac{d \mu_{0} } { d \mu} (\lambda) = \frac{\Trace ( W_\lambda^* W_\lambda ) } {\dim (H_{0,\lambda})}
\end{equation}
for $\mu$-almost all $\lambda$ in the support of the representation $\pi_0$.

The main observation of this paper, which, aside from some functional-analytic details, is very simple,  is that even when $\pi_0$ is only asymptotically contained in $\pi$, we can \emph{still} derive a version of  the formula \eqref{eq-radon-nikodym-fmla} in essentially the same way, if we  \emph{assume} the existence of operators $W_\lambda \colon H_{0,\lambda } {\to} H_{\lambda}$ that \emph{asymptotically} decompose $W$ in the sense that
\begin{equation}
  \label{eq-w-lambda-asymptotics}
  \lim_{t\to +\infty} 
  W_\lambda (U_t v)_{0,\lambda}
  -  (W U_t v)_\lambda  
  = 0  .
\end{equation}
See Section~\ref{sec-asymp-related-reps} for a precise account of the assumptions we make.  

As for  \eqref{eq-w-lambda-asymptotics}, it is easiest to understand its meaning   by examining the adjoint operators
\begin{equation*}
  C_\lambda = W_\lambda^*  \colon H_{\lambda } \longrightarrow H_{0,\lambda}  .
\end{equation*}
In the context of the Sturm-Liouville problem, the spaces $H_\lambda $ and $H_{0,\lambda}$ may be understood as   $\lambda$-eigenspaces for the operators $D$ and $D_0$, respectively, and the asymptotic formula \eqref{eq-w-lambda-asymptotics} simply asserts that each eigenfunction of $D$ is mapped by $C_\lambda$ to an eigenfunction of $D_0$ to which it is asymptotic in the sense of \eqref{eq-c-fn-def1}.  This proves the existence of the operators $C_\lambda$ in this context, and also computes the  trace in \eqref{eq-radon-nikodym-fmla} in terms of $|c(\lambda)|^2$.  Weyl's formula in Theorem~\ref{thm-weyl} is an immediate consequence.\footnote{To be accurate, the Radon-Nikodym derivative formula only determines $\mu$ on the positive part of the spectrum because the necessary assumptions on $W_\lambda$ only hold there. A separate argument is required for the nonpositive spectrum; see Section \ref{sec-non-positive}.}

Other interesting examples of asymptotic containment of representations come from representation theory.  In brief, if $G$ is a real reductive group with maximal compact subgroup $K$, and if $P=M_\mathfrak{p}A_\mathfrak{p}N_\mathfrak{p}$ is a minimal parabolic subgroup, then the representation of $C^*(G)$ on $L^2 (G/M_\mathfrak{p}N_\mathfrak{p})$ is asymptotically contained in the representation of $C^*(G)$ on $L^2 (G/K)$.  The case of $SL(2,\R)$ is illustrated in Figure~\ref{fig-orbit}.  The figure should make it clear that the   asymptotic containment in this example has a very geometric origin.
\begin{figure}[ht]
  \centering \includegraphics[scale=0.3]{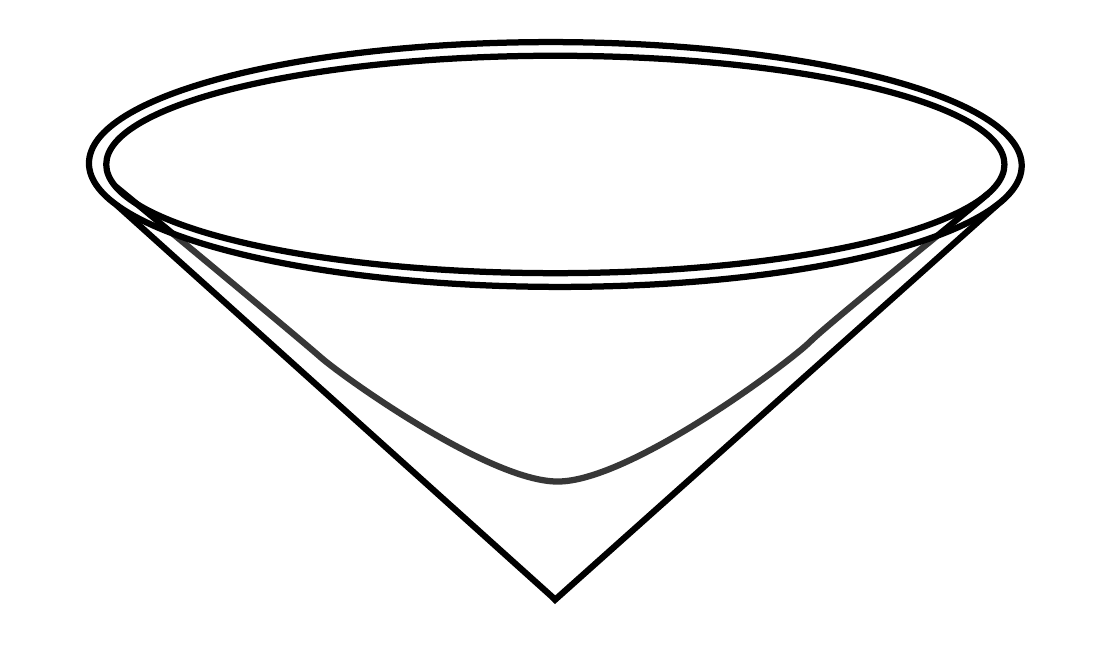}
  \caption{The homogeneous spaces $G/K$ and $G/M_\mathfrak{p}N_\mathfrak{p}$ for the reductive group $G=SL(2,\R)$, realized as coadjoint orbits.   }
  \label{fig-orbit}
\end{figure}

To fit this example within the framework of this paper we take the $C^*$-algebra  $A$ in our asymptotic containment  to be the commutative $C^*$-subalgebra of $C^*(G)$ that is generated by the $K$-bi-invariant compactly supported smooth functions on $G$.   It is naturally represented on the Hilbert spaces  $H$ and $H_0$ of  $K$-fixed vectors within $L^2 (G/K)$ and $L^2 (G/M_\mathfrak{p}N_\mathfrak{p})$, respectively.   The $K$-invariant functions on $G/K$ and $G/M_\mathfrak{p}N_\mathfrak{p}$ can be identified with functions on $A_\mathfrak{p}^{+}$ and $A_\mathfrak{p}$ respectively, where $A_\mathfrak{p}^{+}$ is a dominant chamber in $A_\mathfrak{p}$.  A suitable operator $W$ may be defined using restriction of functions from $A_\mathfrak{p}$ to $A_\mathfrak{p}^+$ (followed by a translation   away from the walls of   $A_\mathfrak{p}^{+}$ to make $W$ bounded).  The minimal parabolic is defined by a one-parameter subgroup of $A$, and right-translation on  $G/M_\mathfrak{p}N_\mathfrak{p}$ by this one-parameter group gives the necessary one-parameter unitary group on $H_0$.
The counterpart of Weyl's theorem in this example is Harish-Chandra's Planch\-erel theorem for spherical functions. See \cite{Tan19}.

In fact the present paper grew out of a project in noncommutative geometry   involving the Plancherel formula \cite{CCH16,CH16}.
The reader is referred to \cite{vandenBan08} for an interesting and thorough discussion of the relation between Weyl's theorem and harmonic analysis on symmetric spaces.   

After this paper was written the authors were lucky to enjoy a very stimulating conversation with Joseph Bernstein, who pointed out that he had obtained very similar results in unpublished work from the 1980's (see the final remarks in \cite[Sec.\;0.2]{Bernstein88} for hints of this). 
Some of the spectral-theoretic methods from \cite{SV12}, which studies harmonic analysis on $p$-adic spherical varieties, are also very closely related to the methods of this paper.  See especially Section 8 of \cite{SV12}. 


Here is a brief outline of the present paper. We shall discuss asymptotic containment of representations in Section~\ref{sec-asymp-related-reps}. The main result is \ref{thm-mu-in-terms-of-mu-zero}.  We shall apply our method to the positive, continuous spectrum of Sturm-Liouville operators in Section~\ref{sec-SL}, and we shall briefly address the nonpositive  spectrum   in Section~\ref{sec-non-positive}.  We shall consider not only the operators discussed in this introduction, but also a nontrivial example related to the representation theory of $SL(2,\R)$. In an appendix we quickly review the Weyl-Kodaira approach to Theorem~\ref{thm-weyl} for the sake of comparison.

\section{Asymptotic Containment of Representations}

\label{sec-asymp-related-reps}

In this section we shall describe  our  approach to Weyl's theorem.  We shall formulate the method in fairly general terms, applicable to examples beyond Weyl's theorem, although we shall not strive for the upmost generality in the assumptions that we make.  

Let $A$ be a separable, commutative $C^*$-algebra with Gelfand spectrum $\Lambda$, so that of course
$
  A \cong C_0(\Lambda)
$ for some locally compact space $\Lambda$.
We shall view elements of $A$ as continuous functions on $\Lambda$ without further comment.

Let us suppose that we are given two non-degenerate representations of $A$ on separable Hilbert spaces:
\[
  \pi \colon A \longrightarrow B(H) \quad \text{and}\quad \pi_0\colon A \longrightarrow B(H_0) .
\]
We shall assume that $\pi_0$ is asymptotically contained in $\pi$, as in Definition~\ref{def-asymptotic-containment}, with unitary group $\{ U_t\}$ on $H_0$ and operator $W \colon H_0 \to H$ as described in the definition.  Our analysis of the relation between $\pi$ and $\pi_0$ will be based on following formula, which is an immediate consequence of \eqref{eq-asymptotic relation-intro}.

\begin{lemma}
  \label{lemma-asymptotic-formula}
  If $a \in A$ and if $g,h\in H_0$, then
  \[
    \pushQED{\qed} \langle g, \pi_0(a )h \rangle_{H_0} = \lim_{T\to +\infty} \frac 1 T \int _0^T \langle WU_{t} g, \pi(a) WU_t h\rangle_{H} \, dt . \qedhere \popQED
  \]
\end{lemma}

We shall make the following assumptions concerning the representation $\pi_0$; in the case of Sturm-Liouville operator they will be easy to verify using Fourier analysis.

\begin{assumption}
  \label{assumptionA}
  We shall suppose that we are given:
  \begin{enumerate}[\rm (i)]
    \item An open subset $ \Lambda_0\subseteq \Lambda $ and a locally trivial continuous field of finite-dimen\-sional Hilbert spaces $\{ H_{0,\lambda} \}_{\lambda\in \Lambda_0} $ over $\Lambda _0$ (or in other words a Hermitian vector bundle).

    \item A dense subspace $ \HH_0\subseteq H_0 $ and a linear map $h\mapsto \{ h_{0,\lambda}\}$ from $\HH_0$ into the continuous sections of $\{ H_{0,\lambda}\}$ such that
    \[
      H_{0,\lambda} = \{ \, h_{0,\lambda}\, : \, h\in \HH_0\,\}
    \]
    for every $\lambda \in \Lambda _0$.
    \item A Borel measure $\mu_0$ on $\Lambda_0$ such that $ \langle h_{0,\lambda} , g_{0,\lambda} \rangle$ is a $\mu_0$-integrable function of $\lambda $, for every $h,g\in \HH_0$, and such that
    \[
      \bigl\langle h, \pi_0(a ) g \bigr\rangle_{H_0} = \int _{\Lambda _0} \bigl \langle h_{0,\lambda} , g_{0,\lambda} \bigr \rangle _{H_{0,\lambda}} a (\lambda) \, d\mu_0(\lambda)
    \]
    for every $h,g\in \HH_0$ and every $a\in A$.
  \end{enumerate}
\end{assumption}

\begin{assumption}
  \label{assumptionB}
  We shall assume that the action of the one-parameter unitary group $\{U_t\}$ on the Hilbert space $H_0$ maps the subspace $\HH_0$ into itself, and that the continuous field $\{ H_{0,\lambda}\} _{\lambda \in \Lambda_0}$ carries a continuous, unitary action $\{ U_{t,\lambda}\} $ of $\R$ such that
  \[
    (U_th) _{0,\lambda} = U_{t,\lambda} h_{0,\lambda}
  \]
  for every $h\in \HH_0$ and every $\lambda\in \Lambda_0$.
\end{assumption}

Next, we shall make assumptions on the representation $\pi$ that are similar to those in Assumption~\ref{assumptionA}, except that we shall in addition assume that $\pi$ has \emph{multiplicity one}: the fibers in the field of Hilbert spaces that decomposes $\pi$ have dimension one.  This assumption isn't altogether necessary (finite-dimen\-sionality would suffice), but it simplifies the statements of the results that follow, along with their proofs, and it is satisfied in the situations of interest to us.  Here are the details.

\begin{assumption}
  \label{assumptionC}
  We shall suppose that there are given:
  \begin{enumerate}[\rm (i)]
    \item A locally trivial continuous field of \emph{one}-dimensional Hilbert spaces $\{ H_{\lambda} \}_{\lambda\in \Lambda} $ over $\Lambda$ (that is, a Hermitian line bundle over $\Lambda$).
    \item A dense subspace $ \HH \subseteq H $ such that if $h\in \HH_0$ then $W U_t h\in \HH$ for all $t\gg 0$.
    \item A family of linear maps $\varepsilon_\lambda\colon h\mapsto  h_{\lambda}$ from $\HH$ into   $ H_{\lambda}$ so that $\{h_\lambda\}$ is a continuous section,  and a Borel measure on $\Lambda$ such that $ \langle h_{\lambda} , g_{\lambda} \rangle$ is a $\mu$-integrable function of $\lambda $, for every $h,g\in \HH$, and such that
    \[
      \bigl\langle h, \pi(a ) g \bigr \rangle_{H} = \int _{\Lambda } \bigl \langle h_{\lambda} , g_{\lambda} \bigr \rangle _{H_{\lambda}} a (\lambda) \, d\mu(\lambda)
    \]
    for every $h,g\in \HH$ and every $a\in A$.
  \end{enumerate}
\end{assumption}

Finally, we shall make the following assumption concerning the asymptotic relation between the fields $\{ H_\lambda\}$ and $\{ H_{0,\lambda} \}$. As we noted in the introduction, and as we shall see clearly in the next section, in the Sturm-Liouville context this means that the operator $C_\lambda$ below maps each $\lambda$-eigen\-function for $D$ to a $\lambda$-eigenfunction for $D_0$ to which it is asymptotic.

\begin{assumption}
  \label{assumptionD}
  We shall assume that there is given a continuous family of \emph{injective} linear maps
  \[
    C_\lambda \colon H_\lambda \longrightarrow H_{0,\lambda}\qquad (\lambda \in \Lambda _0)
  \]
  with the property that if $h$ belongs to $\HH_0$, and if $\{v_\lambda\} $ is a continuous section of $\{ H_\lambda \}$, and $K$ is a compact subset of $\Lambda _0$, then
  \begin{equation*}
    \lim_{t\to +\infty}  \sup_{\lambda \in K} \left | \bigl \langle C_\lambda v_\lambda,  (U_th) _{0,\lambda}\bigr \rangle_{H_{0,\lambda}} - \bigl \langle v_\lambda ,(W U_t h)_{\lambda} \bigr \rangle_{H_\lambda} \right | = 0 .
  \end{equation*}
\end{assumption}

Using the four assumptions listed above we shall prove:
\begin{theorem}
  \label{thm-mu-zero-in-terms-of-mu}
  The measure $\mu_0$ is absolutely continuous with respect to $\mu$ on the open set $\Lambda_0$, with Radon-Nikodym derivative
  \[
    \frac{d \mu_0 }{d \mu} (\lambda)= \frac{\operatorname{Trace}(C_\lambda ^*C_\lambda )}{\dim (H_{0,\lambda})} .
  \]
\end{theorem}

Here is the proof in outline.  We are assuming that
\begin{equation}
  \label{eq-old-integral-fmla}     
  \langle h , \pi_0(\varphi) h \rangle_{H_0} 
  =      
  \int_{\Lambda_0} 
  \bigl \langle h_{0,\lambda}  , h _{0,\lambda}\bigr\rangle_{H_{0,\lambda}} \,\varphi(\lambda) \, d\mu_0 (\lambda) .
\end{equation}  
We shall obtain from our remaining assumptions a new integral formula, namely
\begin{equation}
  \label{eq-new-integral-fmla}
  \langle h , \pi_0(\varphi) h \rangle_{H_0} 
  = 
  \int_{\Lambda_0} 
  \bigl \langle h_{0,\lambda} , \operatorname{Av}\bigl [ C^{\vphantom{*}}_\lambda  C_\lambda ^* \bigr ]h_{0,\lambda} 
  \bigr \rangle_{H_{0,\lambda}}
  \varphi(\lambda) d\mu(\lambda) ,
\end{equation}
where the operator $ \operatorname{Av} [ C^{\vphantom{*}}_\lambda C_\lambda ^* ] \colon H_{0,\lambda}\to H_{0,\lambda}$ is defined by the averaging formula
\[
  \operatorname{Av}\bigl [ C^{\vphantom{*}}_\lambda C_\lambda ^* \bigr ] = \lim_{T \to \infty}\frac{1}{T} \int _0^T U_{-t,\lambda } C^{\vphantom{*}}_\lambda C_\lambda ^* U_{t,\lambda} \, dt
\]
(since we are dealing here with operators on the finite-dimensional space $H_{0,\lambda}$, the limit certainly exists).  At this point, we can appeal to the following uniqueness result for spectral decompositions:

\begin{lemma}
  \label{lemma-unique-dir-int}
  Let $\{T_{\lambda}\}$ be a measurable field of positive operators on $\{ H_{0,\lambda}\}_{\lambda\in \Lambda _0}$.  Suppose that
  \begin{equation*}
    \int_{\Lambda_0} \bigl \langle h_{0,\lambda}  , T_\lambda h_{0,\lambda}  \bigr \rangle_{H_{0,\lambda}}\, \varphi (\lambda) \, d\mu (\lambda)
    = \int_{\Lambda_0} \bigl \langle h_{0,\lambda}  , h _{0,\lambda}\bigr\rangle_{H_{0,\lambda}} \,\varphi(\lambda) \, d\mu_0 (\lambda)
  \end{equation*}
  for every $h\in \HH_0$ and every continuous and compactly supported function $\varphi$.  Then $T_\lambda$ is a scalar multiple of the identity for $\mu$-almost all $\lambda\in \Lambda_0$. In addition $\mu_0$ is absolutely continuous with respect to $\mu$ on $\Lambda_0$ and
  \[
    T_\lambda = \frac{d\mu_0 }{d \mu} (\lambda) \cdot I_{H_{0,\lambda}}
  \]
  $\mu$-almost everywhere on $\Lambda_0$.
\end{lemma}

\begin{proof}
  For each point $\lambda _0 \in \Lambda_0$ there exists $h\in \HH_0$ for which the section $h_{0,\lambda} $ is nonzero at $\lambda_0$. It follows immediately from the uniqueness part of the Riesz representation theorem that $\mu_0$ is absolutely continuous with respect to $\mu$ near $\lambda_0$ with Radon-Nikodym derivative
  \[
    \frac{d\mu_0 }{d \mu} (\lambda)= \frac{\bigl \langle h_{0,\lambda} , T_\lambda h_{0,\lambda} \bigr \rangle_{H_{0,\lambda}}}{\bigl \langle h_{0,\lambda} , h_{0,\lambda} \bigr \rangle_{H_{0,\lambda}}} .
  \]
  Since the derivative is independent of $\{h_{0,\lambda} \}$ this implies that
  \[
    T_\lambda = \frac{d\mu_0 }{d \mu} (\lambda) \cdot I_{H_{0,\lambda}} ,
  \]
  almost everywhere, as required.
\end{proof}

Returning to the proof of Theorem~\ref{thm-mu-zero-in-terms-of-mu}, Lemma~\ref{lemma-unique-dir-int} tells us that the operator $\operatorname{Av}[ C^{\vphantom{*}}_\lambda C_\lambda ^* ]$ is a scalar multiple of the identity for $\mu$-almost-all $\lambda\in \Lambda _0$.  The computation
\[
  \Trace\left ( \operatorname{Av}\bigl [ C^{\vphantom{*}}_\lambda C_\lambda ^* \bigr ] \right ) = \Trace (C^{\vphantom{*}}_\lambda C_\lambda ^*) = \Trace (C^*_\lambda C^{\vphantom{*}}_\lambda ),
\]
determines the multiple, and the theorem follows.  So it remains to establish the integral formula \eqref{eq-new-integral-fmla}:

\begin{lemma}
  If $h\in \HH_0$ and if $\varphi$ is a continuous and compactly supported function on $\Lambda_0$, then
  \begin{equation*}
    \langle h , \pi_0(\varphi) h \rangle_{H_0} 
    = 
    \int_{\Lambda_0} 
    \bigl \langle h_{0,\lambda} , \operatorname{Av}\bigl [ C^{\vphantom{*}}_\lambda  C_\lambda ^* \bigr ]h_{0,\lambda} 
    \bigr \rangle_{H_{0,\lambda}}
    \varphi(\lambda) d\mu(\lambda) .
  \end{equation*}
\end{lemma}

\begin{proof}  
  It suffices to prove this formula for all functions $\varphi$ that are supported on compact sets $K\subseteq \Lambda _0$ over which the field $\{ H_\lambda\}$ is trivializable, and so we shall assume that here. In addition we shall use the notation
  \[ W_t = W U_t \colon H_0 \longrightarrow H.
  \]
  According to Lemma~\ref{lemma-asymptotic-formula},
  \begin{equation}
    \label{eq-asymp-integral0}
    \langle h , \pi_0(\varphi) h \rangle_{H_0} = \lim_{T\to +\infty} \frac{1}{T} \int_0^T \bigl \langle W_t h ,\pi( \varphi )W_{t}h \bigr \rangle_{H} d t .
  \end{equation}
  Now use Assumption~\ref{assumptionC} to write the integrand in the right hand side of \eqref{eq-asymp-integral0} as
  \begin{equation*}
    \bigl \langle  W_t h ,    \pi(\varphi)  W_{t}h \bigr \rangle_H 
    =  \int_{\Lambda_0}   \bigl \langle ( W_t h)_\lambda, ( W_t h)_\lambda  
    \bigr \rangle_{H_\lambda}  \, \varphi(\lambda)   d \mu(\lambda) .
  \end{equation*}
  Since the space $H_\lambda$ are one-dimensional, we can write
  \begin{equation*}
    \bigl \langle ( W_t h)_\lambda, ( W_t h)_\lambda \bigr \rangle_{H_\lambda} = \bigl \langle ( W_t h)_\lambda,v_\lambda \bigr \rangle_{H_\lambda}\cdot \bigl\langle v_\lambda, ( W_t h)_\lambda \bigr \rangle_{H_\lambda} ,
  \end{equation*}
  where $\{ v_\lambda\}$ is a continuous section of $\{ H_\lambda\}$ with $\|v_\lambda \|_{H_\lambda}=1$ for all $\lambda \in K$.  So
  \begin{multline}
    \label{eq-asymp-integral-revised}
    \langle h , \pi_0(\varphi) h \rangle_{H_0} \\ = \lim_{T\to +\infty} \frac{1}{T} \int_0^T \int_{\Lambda_0} \bigl \langle ( W_t h)_\lambda,v_\lambda \bigr \rangle_{H_\lambda}\cdot \bigl\langle v_\lambda, ( W_t h)_\lambda \bigr \rangle_{H_\lambda} \, \varphi(\lambda) d \mu(\lambda) d t .
  \end{multline}

  Consider now the difference
  \begin{multline}
    \label{eq-difference-to-zero}
    \bigl \langle (U_{t} h)_{0,\lambda} ,C_\lambda v_\lambda \bigr \rangle_{H_{0,\lambda}} \cdot \bigl\langle C_\lambda v_\lambda, (U_{t} h)_{0,\lambda} \bigr \rangle_{H_{0,\lambda}}
    \\
    - \bigl \langle ( W_t h)_\lambda,v_\lambda \bigr \rangle_{H_\lambda}\cdot \bigl\langle v_\lambda, ( W_t h)_\lambda \bigr \rangle_{H_\lambda} ,
  \end{multline}
  which we can write as
  \begin{multline*}
    \bigl \langle (U_{t} h)_{0,\lambda} ,C_\lambda v_\lambda \bigr \rangle_{H_{0,\lambda}} \left [ \bigl\langle C_\lambda v_\lambda, (U_{t} h)_{0,\lambda} \bigr \rangle_{H_{0,\lambda}} - \bigl\langle v_\lambda, ( W_t h)_\lambda \bigr \rangle_{H_\lambda}
    \right ] \\
    + \left [ \bigl \langle (U_{t} h)_{0,\lambda},C_\lambda v_\lambda \bigr \rangle_{H_{0,\lambda}} - \bigl \langle ( W_t h)_\lambda,v_\lambda \bigr \rangle_{H_\lambda} \right ] \bigl\langle v_\lambda, ( W_t h )_\lambda \bigr \rangle_{H_\lambda} .
  \end{multline*}
  The terms in the square brackets converge to $0$, as $t\to +\infty$, uniformly in $\lambda \in K$.  In addition, since
  \[
    \bigl | \bigl \langle (U_{t} h)_{0,\lambda} ,C_\lambda v_\lambda \bigr \rangle_{H_{0,\lambda}} \bigr | = \bigl | \bigl \langle U_{t,\lambda} h_{0,\lambda} ,C_\lambda v_\lambda \bigr \rangle_{H_{0,\lambda}} \bigr | \le \| h_\lambda\|\cdot \| C_\lambda v_\lambda\|,
  \]
  we see that $ \langle (U_{t} h)_{0,\lambda} ,C_\lambda v_\lambda \rangle_{H_{0,\lambda}}$ is uniformly bounded in $t$ and $\lambda \in K$.  It follows from this and Assumption~\ref{assumptionD} that $\langle ( W_t h)_\lambda,v_\lambda \rangle_{H_\lambda}$ is uniformly bounded too.  So the expression \eqref{eq-difference-to-zero} converges to zero as $t\to \infty$, uniformly in $\lambda \in K$.

  Observe next that
  \begin{multline*}
    \bigl \langle (U_{t} h)_{0,\lambda} ,C_\lambda v_\lambda \bigr \rangle_{H_{0,\lambda}}\cdot \bigl \langle C_\lambda v_\lambda , (U_{t} h)_{0,\lambda} \bigr \rangle_{H_{0,\lambda}}
    \\
    \begin{aligned}
      & =
      \bigl \langle U_{t,\lambda }  h _{0,\lambda} ,C_\lambda  v_\lambda \bigr \rangle_{H_{0,\lambda}} \cdot \bigl\langle C_\lambda  v_\lambda, U_{t,\lambda }  h_{0,\lambda}\bigr \rangle_{H_{0,\lambda}}  \\
      & = \bigl \langle h_{0,\lambda} ,U_{-t,\lambda } C_\lambda C_\lambda ^* U_{t,\lambda } h_{0,\lambda}\bigr \rangle_{H_{0,\lambda}} .
    \end{aligned}
  \end{multline*}
  It follows from our analysis of \eqref{eq-difference-to-zero} that the inner integral in \eqref{eq-asymp-integral-revised} is asymptotic to the integral
  \begin{equation*}
    \int_{\Lambda _0} \bigl \langle  h_{0,\lambda},U_{-t,\lambda } C_\lambda  C_\lambda ^* U_{t,\lambda }  h_{0,\lambda}  \bigr \rangle_{H_{0,\lambda}}  \, \varphi(\lambda) \, d \mu(\lambda)  
  \end{equation*}
  (that is, the difference converges to zero as $t\to +\infty$). As a result,
  \begin{multline*}
    \langle h , \pi_0(\varphi) h \rangle_{H_0}
    \\
    = \lim_{T\to +\infty} \frac{1}{T} \int_0^T \Bigl ( \int_{\Lambda_0} \bigl \langle h_{0,\lambda},U_{-t,\lambda } C_\lambda C_\lambda ^* U_{t,\lambda} h_{0,\lambda} \bigr \rangle_{H_{0,\lambda}} \, \varphi(\lambda)\, d\mu(\lambda) \Bigr) \, d t ,
  \end{multline*}
  It now follows from Fubini's theorem, that
  \begin{multline*}
    \langle h , \pi_0(\varphi) h \rangle_{H_0}
    \\
    = \lim_{T\to +\infty} \int_{\Lambda _0} \Bigl (\frac{1}{T} \int_0^T \bigl \langle h_{0,\lambda} ,U_{-t,\lambda} C_\lambda C_\lambda ^* U_{t,\lambda} h_{0,\lambda} \bigr \rangle_{H_{0,\lambda}} \, d t \Bigr ) \varphi(\lambda) d\mu(\lambda) .
  \end{multline*}
  The integral in the parentheses is uniformly bounded in $T$.  Therefore we can interchange the limit as $T\to {+}\infty$ and the integral over $\Lambda _0$ to obtain \eqref{eq-new-integral-fmla}, as required.
\end{proof}

Theorem~\ref{thm-mu-zero-in-terms-of-mu} gives a formula for the measure $\mu_0$ in terms of the measure $\mu$.  But since our goal is to obtain information about the measure $\mu$, we should invert this formula:

\begin{theorem}
  \label{thm-mu-in-terms-of-mu-zero}
  The measue $\mu$ is absolutely continuous with respect to the measure $\mu_0$ on $\Lambda_0$, and the Radon-Nikodym derivative of $\mu$ with respect to $\mu_0$ on $\Lambda_0$ is
  \[
    \pushQED{\qed} \frac{d\mu\,}{d \mu_0}(\lambda) = \frac{\dim (H_{0,\lambda})}{\operatorname{Trace}(C_\lambda ^*C_\lambda ^{\phantom{'}})} .\qedhere \popQED
  \]
\end{theorem}

\section{Sturm-Liouville Operators}
\label{sec-SL}

In this section we shall apply the approach of Section~\ref{sec-asymp-related-reps} to Sturm-Liouville operators on the half-line.  So let
\begin{equation*}
  D = - \frac{d\,\,}{dx} \cdot p(x) \cdot \frac{d\,\,} {dx} + q(x) ,
\end{equation*}
where the coefficient functions $p(x)$ and $q(x)$ are smooth and real-valued on $(0,\infty)$, and where $p(x)$ is everywhere positive.  We shall assume that $D$ is a self-adjoint operator on some domain that includes $C_c^\infty (0,\infty)$.  

We shall study the following examples (which may be generalized considerably).

 \begin{example}
\label{ex-simple-sturm-liouville}
If $p$ and $q$ are in fact smooth on $[0,\infty)$ and eventually constant,  with $p$ positive,  as in Section~\ref{sec-introduction}, then $D$ is  essentially self-adjoint on the domain of smooth, compactly supported functions on $ [0,\infty) $ that vanish at $0$.
\end{example}

\begin{example}
\label{ex-sl2-casimir} 
If $G{=}SL(2,\R)$ and $K{=}SO(2)$, then the symmetric space $G/K$ may be identified with the hyperbolic plane (with $G$ acting as isometries on the plane).   The Laplace-Beltrami operator $\Delta$ is essentially self-adjoint on the space of smooth and compactly supported functions on $G/K$.  On $K$-invariant functions it acts as 
\[
\Delta = - \frac{d\;\;}{dr^2} - \coth (r) \frac{d\;}{dr} ,
\] 
where $r$ is the  radial coordinate in the polar coordinate system   associated to the action of $K$.  Now identify the $K$-fixed part of $L^2(G/K)$ with $L^2(0,\infty)$   using the radial coordinate and multiplying by   $\sinh(r)^{\frac 12}$ (the latter comes from the formula $d \mathrm{Area} = \sinh(r) dr d\theta$). We obtain an essentially-self adjoint operator 
\[
\Delta  = D+ \tfrac 14 =  -\frac{d\;\;}{dx^2}  -  \tfrac 14 \operatorname{csch}^2(x) + \tfrac 14 
\]
on $L^2 (0,\infty)$ (we have subtracted the term $1/4$ from $\Delta$ with Lemma~\ref{lem-asymp-inclusion} below in mind).
\end{example}

Associated to the unbounded self-adjoint operator $D$ on the Hilbert space $H=L^2 (0,\infty)$ is the functional calculus representation
\begin{gather*}
  \pi \colon C_0(\R) \longrightarrow B(H) \\
  \pi\colon \varphi \longmapsto \varphi(D).
\end{gather*}
We shall compare $\pi$ to the functional calculus representation
\begin{gather*}
  \pi _0 \colon C_0(\R) \longrightarrow B(H_0) \\
  \pi_0 \colon \varphi \longmapsto \varphi(D_0),
\end{gather*}
where $D_0= -d^2/dx^2$ and $H_0 = L^2(-\infty, \infty)$. Here we view $-d^2/dx^2$ is an essentially self-adjoint operator  on the   domain of smooth, compactly supported functions, and we take $D_0$ to be its self-adjoint extension.

Define $U_t\colon H_0 \to H_0$ to be the translation operator
\[
  (U_th)(x) = h(x\!-\!t) .
\]
Obviously each $\varphi(D_0)$ commutes with each $U_t$.  Denote by
\[
  W \colon H_0\longrightarrow H
\]
the orthogonal projection (which restricts functions on $(-\infty,\infty)$ to functions on $(0,\infty)$, of course).  The following computation checks that $\pi_0$ is asymptotically contained in $\pi$, assuming that  the coefficients of $D$ converge to constant values.

\begin{lemma}
  \label{lem-asymp-inclusion}
Assume   that the coefficients of $D$ satisfy
\begin{equation*}
  \lim_{x\to\infty} p(x) = 1,  \quad \lim_{x\to\infty} p'(x) = 0   \quad \text{and} \quad \lim_{x\to \infty} q(x) = 0 .
\end{equation*}
  If $\varphi\in C_0(\R)$, and if $g,h\in L^2 (-\infty,\infty)$, then
  \begin{equation*}
    \lim_{t\to +\infty} \left  [ \bigl\langle  WU_tg,    \varphi(D) WU_t h  \bigr \rangle _{L^2 (0,\infty)}  
      -  \bigl\langle  g, \varphi(D_0)   h \bigr \rangle_{L^2 (-\infty,\infty)}  \right  ]  = 0 .
  \end{equation*}
\end{lemma}

\begin{proof}
  We shall prove that
  \begin{equation}
    \label{eq-asymptotic-strong-convergence}
    \lim_{t\to + \infty} \bigl \|  \varphi (D) W U_t h - WU_t \varphi(D_0)  h \bigr \|  _{L^2 (0,\infty)} = 0
  \end{equation}
  for every $h\in L^2 (-\infty,\infty)$, which will suffice.  The set of all $\varphi\in C_0(\R)$ satisfying \eqref{eq-asymptotic-strong-convergence} is a norm-closed subalgebra of $C_0(\R)$, and it therefore suffices to show that the resolvent functions $\varphi(\lambda) = (\lambda \pm i)^{-1}$ belong to it.  Moreover it suffices to check \eqref{eq-asymptotic-strong-convergence} for each of these two functions $\varphi$ and for a dense set of functions $h$ in $L^2 (-\infty, \infty)$.

  Let $\varphi(x) = (x\pm i)^{-1}$.  We shall calculate the limit \eqref{eq-asymptotic-strong-convergence} when
  \[
    h = (D_0 \pm i I)f\quad \text{and} \quad f\in C_c^\infty (-\infty, \infty) .
  \]
  If $f\in C_c^\infty (-\infty ,\infty )$, and if $t\gg 0$, then $WU_t f $ is a smooth and compactly supported function on $(0,\infty)$, and we compute that
  \begin{equation*}
    \begin{aligned}
      \varphi (D) WU_th - WU_t\varphi(D_0)h
      & = (D\pm iI)^{-1} WU_t (D_0\pm  iI)f  -  WU_t  f  \\
      & = (D\pm iI)^{-1} (D-D_0)W U_t f ,
    \end{aligned}
  \end{equation*}
  where, in the last line, we are regarding $D_0$ as a differential operator acting on the smooth and compactly supported functions on $(0,\infty)$.  Our assumptions on the coefficients of $D$   imply that
  \[
    \lim_{t \to + \infty} \| (D-D_0)W U_t f \| = 0,
  \]
  and so \eqref{eq-asymptotic-strong-convergence} is proved for $\varphi(x) = (x\pm i)^{-1}$, as required.
\end{proof}

Assumptions \ref{assumptionA} and \ref{assumptionB} about the representation $\pi_0$ from the previous section are easily obtained from the Fourier transform
\[
  \widehat h (\xi ) = \int_{-\infty} ^\infty h(x) e^{- i \xi x} \, dx ,
\]
as follows.  To begin, let $ \Lambda _0 = (0,\infty), $ and for $\lambda \in \Lambda _0$ define $H_{0,\lambda}$ to be the two-dimensional vector space of functions on the line spanned by $e^{i\sqrt{\lambda} x}$ and $e^{-i\sqrt{\lambda} x}$. Equip $H_{0,\lambda}$ with the inner product that makes these two functions an orthonormal basis.  The family $\{ H_{0,\lambda}\}_{\lambda > 0}$ obviously forms a continuous field of Hilbert spaces over $\Lambda _0$ with constant and finite fiber dimension.

Now let $\HH_0$ be space of smooth and compactly supported functions in $H_0$.  The Fourier transform associates to each $h\in \HH_0$ a continuous section $\{ h _{0,\lambda} \}$ of the continuous field, namely
\[
  h_{0,\lambda} = \widehat h ( \sqrt{\lambda}) e^{ i \sqrt{\lambda} x} + \widehat h(- \sqrt{\lambda}) e^{- i \sqrt{\lambda} x} .
\]
Moreover it follows from Plancherel's formula that
\[
  \langle h, \varphi(D_0) g\rangle _{L^2 (-\infty, \infty)} = \int_{\Lambda _0}\langle h_{0,\lambda} , g_{0,\lambda} \rangle_{H_{0,\lambda}}\,\varphi(\lambda) \, d\mu_0(\lambda),
\]
where
\begin{equation}
  \label{eq-mu-0-fmla}
  d\mu_0( \lambda ) = \frac{1}{4 \pi} \frac{d\lambda}{\sqrt{\lambda}}.
\end{equation}
So Assumption~\ref{assumptionA} is satisfied.  The unitary actions
\[
  U_{t,\lambda} \colon a e^{ i \sqrt{\lambda} x} +b e^{- i \sqrt{\lambda} x} \longmapsto e^{ -i \sqrt{\lambda} t}a e^{ i \sqrt{\lambda} x} +e^{ i \sqrt{\lambda} t}b e^{- i \sqrt{\lambda} x}
\]
on the fibers $H_{0,\lambda}$ decompose the translation action on $L^2(-\infty,\infty)$, as in Assumption~\ref{assumptionB}.

Let us turn now to the representation $\pi$ of $C_0(\R)$. General theory guarantees that $\pi$ has a measurable direct integral decomposition
\begin{equation}
  \label{eq-direct-integral}
  L^2 (0,\infty)  \cong  \int^\oplus_\R H_\lambda \, d\mu(\lambda) .
\end{equation}
This means that there exists:
\begin{enumerate}[\rm (i)]
  \item A Borel-measurable field of Hilbert spaces, $\{ H_\lambda\} _{\lambda \in \R}$, as in \cite[Part II, Chapter 1]{Dixmier81}.
  \item A Borel measure $\mu$ on $\R$.
  \item A unitary isomorphism from $L^2 (0,\infty)$ to the Hilbert space of square-inte\-grable sections of the measurable field, $h\mapsto \{ h_\lambda \}_{\lambda \in \R}$,
  under which the representation $\pi$ corresponds to the representation of $C_0(\R)$ on square-integrable sections by pointwise multiplication. Thus if $g\in H$ and if $\varphi \in C_0(\R)$, then
  \[
   (\pi(\varphi)g)_\lambda = \varphi(\lambda)  g_\lambda
  \]
  for $\mu$-almost every $\lambda \in \R$.
\end{enumerate}
See \cite[Part II, Chapter 6, Theorem 2]{Dixmier81}. 
We need to upgrade this  measurable decomposition to a continuous decomposition, as required by Assumption~\ref{assumptionC}.  We don't know the full extent to which this is possible,  but the  \emph{Gelfand-Kostyuchenko   method}, which we shall now review, handles the examples of concern to us.  (See \cite[Section 1]{Bernstein88} for a concise account of the Gelfand-Kostyu\-chenko method, as well as applications that are closely related to those in this paper.)

The inclusion of the topological vector space $C_c^\infty (0,\infty)$ into $L^2 (0, \infty)$ factors through a Hilbert-Schmidt operator.  That is, there is a commuting diagram
\begin{equation}
\label{eq-hilbert-schmidt-factorization}
  \xymatrix{ C_c^\infty (0,\infty) \ar[rr]^{\text{inclusion}} \ar[dr]_{\text{continuous}}& & L^2 (0,\infty) \\ &K\ar[ur]_{\text{Hilbert-Schmidt}} & }
\end{equation}
where $K$ is a Hilbert space.  This has the following consequence:

\begin{lemma}[See for example {\cite[Chapter VII, Section 1]{Maurin67}}]
\label{lem-g-k-lemma}
 For all $\lambda\in \R$ there exist continuous linear operators
  \begin{equation}
    \label{eq-epsilon-maps}
    \varepsilon_\lambda\colon  C_c^\infty (0,\infty )  \stackrel{ }\longrightarrow H_\lambda
  \end{equation}
  such that if $h\in C_c^\infty (0,\infty)$, and if $\{h_\lambda\}_{\lambda \in \R}$ is the associated square-integrable section of $\{ H_\lambda\}_{\lambda\in \R}$, then $ h_\lambda = \varepsilon_\lambda(h) $ for $\mu$-almost every $\lambda \in \R$.  \qed
\end{lemma}

Since $C_c^\infty (0,\infty)$ is dense in the Hilbert space $L^2 (0,\infty)$, the maps $\varepsilon_\lambda$  have dense range for $\mu$-almost every $\lambda$.  The adjoint operators
\begin{equation}
  \label{eq-adjoint-epsilon}
  \varepsilon ^* _\lambda \colon  H_\lambda^*  \longrightarrow C_c^\infty (0,\infty)^*
\end{equation} 
are therefore injective for $\mu$-almost every $\lambda$.  That is, for almost every $\lambda{\in} \R$ the map $\varepsilon_\lambda ^*$ is defined and embeds $H_\lambda^*$ into the space of distributions on $\R$.

Keeping in mind the Hilbert space isomorphism $ H^*_\lambda \cong \overline{H_\lambda}$, it follows from   Lemma~\ref{lem-g-k-lemma} and the definitions  that if $h\in C_c^\infty (0,\infty)$, and if $\{ v_\lambda \}_{\lambda \in \R}$ is a measurable section of $\{ H_\lambda\}_{\lambda \in \R}$, then
\[
  \langle v_\lambda, h_\lambda\rangle_{H_\lambda} = \langle v_\lambda, \varepsilon_\lambda (h)\rangle_{H_\lambda} = \int_0^\infty \overline{\varepsilon_\lambda ^*(v_\lambda ) }\cdot h ,
\]
for $\mu$-almost every $\lambda \in \R$, where the right-hand integral is the pairing between distributions and test functions.  Using this and the propery (iii) above, we find that if $V_\lambda = \overline{\varepsilon_\lambda ^* (v_\lambda)}$, then
\[
  \int_0^\infty D V_\lambda \cdot h = \int _0^\infty \lambda V_\lambda \cdot h
\]
for $\mu$-almost every $\lambda \in \R$ (the operator $D $ is applied to $V_\lambda $ in the sense of distributions) and since $C_c^\infty (0,\infty)$ is separable it follows that
\[
  D V_\lambda = \lambda V_\lambda
\]
for $\mu$-almost every $\lambda $. Thus for almost every $\lambda$, the morphism $\overline{\varepsilon_\lambda^*}$ embeds
$H_\lambda$ into the space of $\lambda$-eigendistributions for $D$ on $(0,\infty))$.  The latter is $2$-dimensional  and consists of smooth functions on $(0,\infty)$.   

Let us  study the implications of all this for the operators in Example~\ref{ex-simple-sturm-liouville}.

\begin{lemma}
  \label{lemma-mult-one}
 Let $D$ be as in  Example~\textup{\ref{ex-simple-sturm-liouville}}.
  For $\mu$-almost every $\lambda\in \R$ the operator $\overline{\varepsilon_\lambda ^*}$  maps $H_\lambda$ isomorphically to the one-dimensional space of \textup{(}smooth\textup{)} solutions of the differential equation $DG_\lambda = \lambda G_\lambda$ that satisfy the boundary condition $G_\lambda (0) = 0$.

\end{lemma}

\begin{proof}
We can repeat the Gelfand-Kostyuchenko method above using the space of  functions in $C_c^\infty [0,\infty)$ that vanish at $0$ in place of $C_c^\infty(0,\infty)$.
 If $f$ and $g$ belong to this space, then for almost every $\lambda$ we can write
  \begin{multline}
    \label{eq-gelf-kost-wronski}
    \langle \varepsilon_\lambda (Dh), g_\lambda \rangle _{H_\lambda} - \langle \varepsilon_\lambda (h), (Dg)_\lambda \rangle _{H_\lambda}
    \\
    \begin{aligned}
      & =  \langle \varepsilon_\lambda (Dh),  g_\lambda \rangle _{H_\lambda} - \langle \varepsilon_\lambda (h), \lambda  g_\lambda \rangle _{H_\lambda} \\
      & = \int_0^\infty  \overline{(Dh) (x)} G_\lambda  (x)\, dx  -  \int_0^\infty \overline{ h (x) } \lambda   G_\lambda (x) \, dx\\
      & = \int_0^\infty \overline{(Dh) (x)} G_\lambda (x)\, dx - \int_0^\infty \overline{ h (x) } (D G_\lambda) (x) \, dx ,
    \end{aligned}
  \end{multline}
  where $G_\lambda=\overline{\varepsilon_\lambda^*} ( g_\lambda)$. Assume now that in addition $h'(0)=1$.  Calculating the difference of integrals using the fundamental theorem of calculus we find that
  \[
    \int_0^\infty \overline{(Dh) (x)} G_\lambda (x)\, dx - \int_0^\infty \overline{ h (x) } (D G_\lambda) (x) \, dx = p(0) G_\lambda (0) .
  \]
  The top expression in \eqref{eq-gelf-kost-wronski} is an integrable function of $\lambda$, and therefore so is $ G_\lambda (0)$.  If $\varphi$ is any continuous and compactly supported function on $\R$, then by (iii) above the integral of the left-hand side of \eqref{eq-gelf-kost-wronski}, times $\varphi(\lambda)$, is equal to zero, and so
  \[
    \int_{0}^\infty G_\lambda(0) \,\varphi (\lambda ) \, d\mu(\lambda) = 0 ,
  \]
  It follows that $G_\lambda (0) = 0$ for almost every $\lambda$.  The lemma follows from this because the elements $ g_\lambda$ span $H_\lambda$, for almost all $\lambda$.\end{proof}

Now form  the family of    one-dimensional    eigenfunction spaces 
\begin{equation}
\label{eq-cts-field-eigenspace-fiber}
\{ \, F_\lambda\colon [0,\infty){\to} \C\, :\,  DF_\lambda = \lambda F_\lambda , \;\; F_\lambda (0)=0\,\} .
\end{equation}
These assemble to form the fibers of a smooth vector bundle using the usual topology of convergence of smooth functions. Equip each  with the norm $\| F_\lambda\|    = | F_\lambda '(0)|  $  to obtain a continuous field of one-dimensional Hilbert spaces over $\R$ for which $\lambda \mapsto F_\lambda$ is a continuous section if $\lambda\mapsto F'_\lambda(0)$ is continuous.
 
 Lemma~\ref{lemma-mult-one} shows that  for almost every $\lambda$ the morphism $\overline{\varepsilon_\lambda^*}$ is a vector space isomorphism from   the Hilbert space fiber  $H_\lambda$ in the direct integral decompostion  \eqref{eq-direct-integral}   to the fiber \eqref{eq-cts-field-eigenspace-fiber} above. The morphism is not necessarily isometric, but we can remedy this possible shortcoming by changing the  inner products on the $H_\lambda$, and the measure $\mu$, using
\[
  \langle \,\,\,,\,\,\rangle _{H_\lambda} : =   \| v_\lambda \|^{-2}_{H_\lambda}\cdot  \langle \,\,\,,\,\,\rangle _{H_\lambda} \quad\text{and} \quad d\mu (\lambda) : =     \| v_\lambda \|^{2}_{H_\lambda}  \cdot d\mu(\lambda)
\]
where $v_\lambda $ is chosen so that if $F_\lambda =\overline{\varepsilon_\lambda^*}(v_\lambda)$ then $F'_\lambda(0)=1$.  With these changes, we obtain a new direct integral decomposition of the form \eqref{eq-direct-integral}   (the map $h\mapsto \{ h_\lambda \} $ from $L^2 (0,\infty)$ to square-integrable sections is not changed),  and now the morphisms $\overline{\varepsilon^*_\lambda}$ are unitary isomorphisms, for almost every $\lambda$.

Now take
$
  \HH = C_c^\infty (0,\infty)  
$,
and if $h\in \HH$, then according to the definitions, if $F_\lambda'(0)=1$, then
\[
  \overline{\varepsilon^*_\lambda}(h_\lambda) =  \int_0^\infty \overline{F_\lambda(x)} h(x)\, dx \cdot F_\lambda .
\]
for almost every $\lambda$. The right hand side  is a continuous section of the field $\{ H_\lambda\}$ since the function $F_\lambda$ depends continuously (in fact analytically) on $\lambda$.
This verifies Assumption~\ref{assumptionC} for the Sturm-Liouville operators from Section~\ref{sec-introduction}.

As for the Laplace-Beltrami operator from Example~\ref{ex-sl2-casimir}, we can repeat the Gel\-fand-Kostyuchenko analysis, as in Lemma~\ref{lem-g-k-lemma} and the discussion following the lemma, using the space of  smooth, compactly supported, $K$-invariant functions on $G/K$ in place of $C_c^\infty (0,\infty)$, and obtain, for almost every $\lambda$, embeddings of $H_\lambda$ into the $K$-invariant $\lambda$-eigenfunctions of $D$.  But the latter space is actually one-dimensional already (the $\lambda$-eigenfunctions are distinguished from one another by their values at $eK\in G/K$) and the family of all such  eigen\-spaces spaces carries the structure of a continuous field of one-dimensional Hilbert spaces,  since there are explicit formulas for the eigenfunctions that vary smoothly with $\lambda$. See  \cite[Chap.\,2,\,Thms\;1.1\;\&\;1.2]{Helgason72}.  The argument above then handles Assumption~\ref{assumptionC} in this case.

Finally, we need to   verify Assumption~\ref{assumptionD}.  For this purpose we shall     assume a bit more about the coefficients of $D$, namely that
\begin{equation}
  \label{eq-hypotheses-p-and-q2}
  \int_{x_0}^\infty|1- {p(x)^{-1}}| \, d x< \infty \quad \text{and} \quad \int_{x_0}^\infty|q(x)| \, d x< \infty.
\end{equation}
Certainly these conditions hold in our examples. 

\begin{proposition}
  \label{prop-eigenfunction-asymptotics}
  Let $\lambda > 0$ and let $F_\lambda$ be the $\lambda$-eigenfunction of $D$ with $F'(0) =1$. If \eqref{eq-hypotheses-p-and-q2} holds, then there is a unique nonzero $\lambda$-eigenfunction $F_{0,\lambda}$ of $D_0$ such that
  \begin{equation*}
    \lim_{x\to \infty}    \bigl |F_\lambda(x)-F_{0,\lambda}(x)  \bigr |   = 0.    \end{equation*}
  The convergence is uniform over compact  sets of eigenvalues $\lambda$ in $(0,\infty)$.  \qed
\end{proposition}

This is standard in differential equations and we will omit the proof here, but see for example Weyl's paper \cite{Weyl10}).  Of course Proposition \ref{prop-eigenfunction-asymptotics} is obvious for the operators from Example~\ref{ex-simple-sturm-liouville}.

In any case, using Proposition~\ref{prop-eigenfunction-asymptotics} we define injective operators
\begin{equation*}
  C_\lambda  \colon H_\lambda\longrightarrow  H_{0,\lambda}
\end{equation*}
by $C_\lambda \colon F_\lambda \mapsto F_{0,\lambda}$ where $F_\lambda$ and $F_{0,\lambda}$ are as in Proposition \ref{prop-eigenfunction-asymptotics}. If $h$ is a smooth, compactly supported function on $\mathbb{R}$, and if $v_\lambda = F_\lambda$, then
\begin{equation}
  \label{eq-check-assumptionD}
  \bigl \langle C_\lambda  v_\lambda, (U_th)_{0,\lambda}\bigr \rangle_{H_{0,\lambda}} - \bigl \langle v_\lambda ,(W_t h)_{\lambda} \bigr \rangle_{H_\lambda}
  =
  \int_0^\infty (\overline{F_{0,\lambda}}(x)-\overline{F_\lambda}(x))h(x{-}t)dx
\end{equation}
(this formula holds as long as $t$ is large that $h(x{-}t)$ is supported on the positive $x$-axis).  Proposition~\ref{prop-eigenfunction-asymptotics} implies that if $\lambda$ is confined to a compact set in $(0,\infty)$, then the integral converges to zero, uniformly in $\lambda$, as required by Assumption~\ref{assumptionD}.

We arrive therefore the following result, which is Weyl's theorem for the positive spectrum of $D$:
\begin{theorem}
  \label{thm-weyl-redux}
 Let $D$ be one of the operators from Example~\ref{ex-simple-sturm-liouville}.  Let $g$ and $h$ be smooth and compactly supported functions on $[0,\infty)$.  If $0{<}\alpha{<} \beta$, and if $P_{[\alpha,\beta]}$ is the spectral projection for $D$ associated to the interval $[\alpha,\beta]$, then
  \[
    \langle g, P_{[\alpha,\beta]}h\rangle = \frac{1}{4 \pi} \int _\alpha^\beta \langle g, F_\lambda\rangle \langle F_\lambda, h\rangle \frac{1}{|c(\lambda)|^2}\, \frac{ d\lambda }{\sqrt{\lambda}}
  \]
  where $F_\lambda$ is the nonzero  $\lambda$-eigenfunction with $F_\lambda(0)=0$ and $F'_\lambda(0)=1$,  and $c(\lambda)$ is characterized by
  \[
    \lim_{x\to +\infty} \bigl (F_\lambda (x) - c(\lambda ) e^{i\sqrt{\lambda} x} - \overline{c(\lambda ) }e^{-i\sqrt{\lambda} x}\bigr ) = 0
  \]
 \textup{(}the inner products are standard $L^2$-inner products and the integral is absolutely convergent\textup{)}.
\end{theorem}

\begin{proof} 
  We shall compute $\|P_{[\alpha,\beta]} h\|^2 $ (the formula in the statement of the theorem will follow by polarization).  First, according to the definition of a direct integral decomposition,
  \[
    \|P_{[\alpha,\beta]} h\|^2 = \int _\alpha^\beta \| h_\lambda\|_{H_\lambda} ^2 \, d\mu (\lambda) .
  \]
  Now let $\{ v_\lambda\}$ be the section of $\{ H_\lambda \}$ for which $\overline{ \varepsilon^* _\lambda}({v_\lambda}) =  {F_\lambda}$, with $F_\lambda$ as in the statement of the theorem.  Then
  \[
    \int _\alpha^\beta \| h_\lambda\|_{H_\lambda} ^2 \, d\mu (\lambda) 
    = \int _\alpha^\beta \frac{ \bigl | \langle v_\lambda, h_\lambda\rangle _{H_\lambda} \bigr | ^2 } { \langle v_\lambda, v_\lambda\rangle _{H_\lambda}} \, d\mu (\lambda) 
    = \int _\alpha^\beta \frac{| \langle F_\lambda, h_\lambda\rangle_{L^2} | ^2 } { \langle v_\lambda, v_\lambda\rangle_{H_\lambda} } \, d\mu (\lambda) ,
  \]
  and applying Theorem~\ref{thm-mu-in-terms-of-mu-zero} we get
  \[
    \begin{aligned}
      \int _\alpha^\beta \| h_\lambda\|_{H_\lambda} ^2 \, d\mu (\lambda) 
      	& = \int _\alpha^\beta \frac{| \langle F_\lambda, h_\lambda\rangle_{L^2} | ^2 } { \langle v_\lambda, v_\lambda\rangle_{H_\lambda} } \,
      \frac{2 d\mu_0 (\lambda)}{\Trace(C_\lambda ^*C^{\vphantom{*}}_\lambda )}      \\
      & = 2 \int _\alpha^\beta \frac{| \langle F _\lambda, h_\lambda\rangle_{L^2} | ^2} { \langle C_\lambda v_\lambda, C_\lambda v_\lambda\rangle_{H_{0,\lambda}} } \, d\mu_0 (\lambda).
    \end{aligned}
  \]
  It follows from our definition of $C_\lambda $ that this is
  \[
    \int _\alpha^\beta \frac{ | \langle F _\lambda, h_\lambda\rangle_{L^2} | ^2  } { |c(\lambda)|^2 } \, d\mu_0 (\lambda),
  \]
  and the theorem follows from the explicit formula for $\mu_0$ in \eqref{eq-mu-0-fmla}.
\end{proof}

There is a  similar theorem for the operator in Example~\ref{ex-sl2-casimir}.  The only change is that $F_\lambda$  is taken to be the $\lambda$-eigenfunction of $D$ on $(0,\infty)$ corresponding to the $K$-equivariant $\lambda$ eigenfunction 
$G/K{ \to} \C$  of the shifted Laplace-Beltrami operator with value $1$ at $eK$.

\section{Non-Positive Spectrum}
\label{sec-non-positive}

In this section we shall look at  the non-positive part of the spectrum of a Sturm-Liouville operator $D$ of the types considered in the previous section.  The methods of this paper really have nothing to contribute here, and for that reason  we shall be extremely brief.

The value $\lambda{=}0$ belongs to the spectrum of $D$ of any of the operators from Section~\ref{sec-introduction} because the spectrum is closed.  But for the purposes of fully determining the measure $\mu$ we need to determine whether or not $0$ is an eigenvalue of the self-adjoint operator $D$, or in other words whether or not $\mu(\{0\}) > 0$.

The answer is that $\lambda{=}0$ is not an eigenvalue.  For the Sturm-Liouville differential operators from Section~\ref{sec-introduction}, the $\lambda{=}0$ eigenfunctions have the form
\[
F_0(x) = c_1  + c_2 x \qquad (x\gg0) ,
\]
and the only possibility for a \emph{square-integrable} eigenfunction is $c_1{=}c_2{=}0$, in which case $F_0$ is identically zero.  But any eigenfunction for the self-adjoint operator $D$ would in particular be a square-integrable eigenfunction for the differential operator $D$.

 For the Laplace-Beltrami operator from Example~\ref{ex-sl2-casimir} one can employ a similar argument, using a version of  Proposition~\ref{prop-eigenfunction-asymptotics} in place of the simple asymptotic formula for $F_0$ given above.

The negative part of the spectrum for the operators that we discussed in Section~\ref{sec-introduction} needs to be handled differently, since   square-integrable eigenfunctions are certainly possible in this case.  But one can prove, using the same methods that go into the proof of  Proposition~\ref{prop-eigenfunction-asymptotics},  that: 

\begin{proposition}
\label{prop-few-square-integrable-eigenfunctions} If we assume that
  \[
    \int_{1}^\infty e^{\alpha x} |1- {p(x)^{-1}}| \, d x< \infty \quad \text{and} \quad \int_{1}^\infty e^{\alpha x} |q(x)| \, d x< \infty.
  \]
  for some $\alpha > 0$, then the operator $D$ has at most finitely many $L^2$-eigenfunctions satisfying the boundary condition $F_\lambda (0) = 0$. \qed
\end{proposition}

One can say more using perturbation theory.  The operators $D$ from Section~\ref{sec-introduction}  are    relatively compact perturbations of the positive operators
\[
  -d/dx\cdot p(x)\cdot d/dx .
\]
So the negative parts of their spectra consist of  at most countably many eigenvalues, accumulating only at $0$.  Compare \cite[Chapter IV, Theorem 5.35]{Kato76}. But Proposition~\ref{prop-few-square-integrable-eigenfunctions} rules out the possibility of accumulation at $0$. Hence the negative spectra are finite in this case.

As  for the Laplace-Beltrami operator from Example~\ref{ex-sl2-casimir}, it is not difficult to show that $D\ge 0$, so there is no negative spectrum at all.

\section*{Appendix: Review of Kodaira's Approach}
\label{sec-kodaira-revised}

In this appendix we shall   review Weyl's approach to Theorem~\ref{thm-weyl} , as  improved by Kodaira \cite{Kodaira49} (see also \cite{Weyl50}).  Our aim in doing so is to indicate how different it is from the approach taken in the body of this paper. 

Let $D$ be a self-adjoint Hilbert space operator.  If $\alpha{ <} \beta$, and if both $\alpha$ and $\beta$ are absent from the spectrum of $D$, then according to the Riesz functional calculus, the spectral projection for $D$  associated to the interval $(\alpha,\beta)$ is
\begin{equation}
  \label{eq-projection1.5}
  P_{(\alpha,\beta)}   =  \frac{1}{2 \pi i }  \int _{\Gamma} (\nu - D)^{-1} \, d\nu  ,
\end{equation}
where the contour $\Gamma$ is indicated in Figure~\ref{fig:contour}.
\begin{figure}[ht] 
  \centering \includegraphics[scale=0.2]{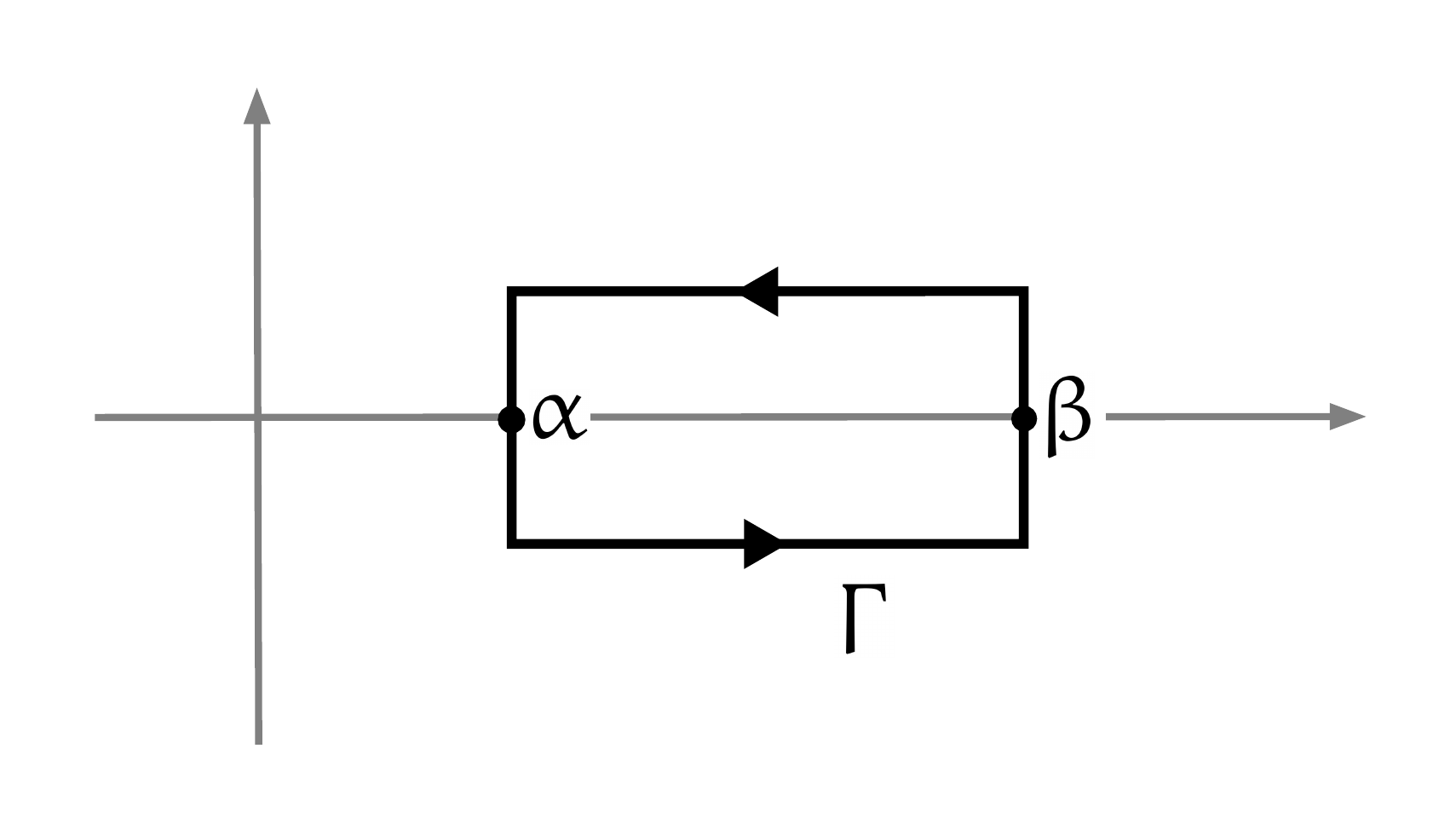}
  \caption{The contour for the integral in \eqref{eq-projection1.5}}
  \label{fig:contour}
\end{figure}
The contributions to the integral in \eqref{eq-projection1.5} from the vertical components of the contour $\Gamma$ decrease to zero in norm as the height of the contour decreases to zero, and so
\begin{equation}
  \label{eq-projection2}
  P_{(\alpha,\beta)}  = \lim_{\varepsilon \to 0+} \frac{1}{2 \pi i }
  \left (   \int _{\alpha - i\varepsilon}^{\beta-i\varepsilon} (\nu - D)^{-1} \, d\nu  
    -  \int _{\alpha + i\varepsilon}^{\beta+i\varepsilon} (\nu - D)^{-1} \, d\nu  \right ) ,
\end{equation}
or equivalently
\begin{equation}
  \label{eq-projection3}
  P_{(\alpha,\beta)} 
  = \lim_{\varepsilon \to 0+} \frac{1}{2 \pi i }
  \int _{\alpha }^{\beta } 
  \, (D {-} \lambda {-} i\varepsilon )^{-1}   -  (D {-} \lambda  {+} i\varepsilon )^{-1}  
  \, d\lambda 
\end{equation}
(these are norm limits).  The integrand on the right-hand side of \eqref{eq-projection3} is uniformly bounded in $\varepsilon{>}0$ and in $\lambda \in \R$, and by approximating a general self-adjoint operator $D$ by operators that do not contain $\alpha$ or $\beta$ in their spectrum, we find that:

\begin{lemma}[Kodaira]
\label{lem-kodaira-lemma}
  The formula \eqref{eq-projection3} holds for \emph{any} self-adjoint operator $D$ and any interval $[\alpha,\beta]$, as long as as $\alpha$ and $\beta$ do not belong to the point spectrum of $D$ \textup{(}the limit in \eqref{eq-projection3} is now a strong limit of a uniformly bounded family of operators\textup{)}. \qed
\end{lemma}

 The formula \eqref{eq-projection3} is of particular value when $D$ is a Sturm-Liouville operator because, as we shall see, the resolvent operators $(D{-}\lambda{\pm}i\varepsilon)^{-1}$ may be computed quite explicitly.  
Let us consider  then
\begin{equation*}
  D = - \frac{d\,\,}{dx} \cdot p(x) \cdot \frac{d\,\,} {dx} + q(x) 
\end{equation*}
where  $p$ and $q$ are smooth, real-valued functions on $(0,\infty)$ (we shall impose further conditions on $p$ and $q$ later on).  Assume  that $D$ defines a self-adjoint operator on $L^2 (0,\infty)$ on a given domain that (i) includes the smooth, compactly supported functions on $(0,\infty)$ and (ii) is invariant under multiplication by smooth functions on $(0,\infty)$ that are locally constant outside of a compact set.

If $\nu  \notin  \R$ (or more generally if $\nu$ belongs to the resolvent set of $D$), then there exist nonzero $\nu$-eigenfunctions $F_\nu$ and $G_\nu$ that vary smoothly with $\nu$, the first of which agrees with an element in $\dom(D)$ near $0$ and the second of which agrees with an element of $\dom (D)$ near $\infty$.  Indeed, if $h$ is any smooth and compactly supported function on $(0,\infty)$, then the function
\[
  f = (D-\nu)^{-1}h
\]
belongs to $\dom (D)$, and moreovoer
\[
  D f = \nu f + h .
\]
It follows that $Df {=} \nu f$ near $0$ and near $\infty$.  Because the set of all $ (D{-}\nu)^{-1}h $ is dense in $L^2 (0,\infty)$, we obtain, for at least some $h$, functions $f$ that are  \emph{nonzero}  near $0$ and near $\infty$. They agree there with functions in $\dom (D)$, and they extend to nonzero $\nu$-eigenfunctions $F_\nu$ and $G_\nu$ on $(0,\infty)$, as required.  

Note that  eigenfunctions $F_\nu $ and $G_\nu$ must be linearly independent, for otherwise they would belong to $\dom (D)$, which is impossible if $\nu \notin \Spec (D)$. 

 Consider now the integral kernel defined by
\begin{equation}
\label{eq-def-of-k-nu}
  k _\nu(x,y) =
  \begin{cases}
    F _\nu(y) G _\nu  (x)  &  x \ge y \\
    F _\nu(x)G _\nu(y) & x \le y .
  \end{cases}
\end{equation}
The associated integral operator $K_\nu$ can certainly be defined on the domain of smooth, compactly supported functions $h$ on $(0,\infty)$, and moreover since
\[
  (K_\nu h)(x) = F_\nu(x) \int _x ^\infty G _\nu(y) h (y)\, dy + G _\nu(x) \int_0^x F _\nu(y) h(y) \, dy .
\]
the range consists of smooth functions in $\dom (D)$.  We compute directly that
\begin{equation}
\label{eq-greens-function-fmla1}
(  D -\nu)K_\nu h   = \Wr(F_\nu,G_\nu) h  ,
\end{equation}
where $\Wr(F_\nu,G_\nu)$ is the  {Wronskian}
\begin{equation*}
  \Wr (F_\nu , G_\nu )(x) =  p (x) \bigl (F'_\nu(x)  G_\nu(x)  - F_\nu(x)  G'_\nu(x) \bigr) .
\end{equation*}
As is well known, this  is a \emph{constant} function of $x{\in} (0,\infty)$; moreover the constant value determines    a   nondegenerate bilinear form on the $2$-dimensional space of $\nu$-eigenfunctions. Since $F_\nu$ and $G_\nu$ are linearly independent, we can therefore normalize them so that 
\begin{equation}
\label{eq-wronskian-normalization1}
\Wr(F_\nu,G_\nu)=1,
\end{equation}
in which case it follows from \eqref{eq-greens-function-fmla1} that 
\begin{equation}
\label{eq-greens-function-fmla2}
 (D {-} \nu)^{-1} h = K_\nu h 
\end{equation}
for all smooth and compactly supported functions $h$ on $(0,\infty)$.

In order to apply \eqref{eq-greens-function-fmla2} to the limit formula \eqref{eq-projection3} we shall make the following additional assumptions concerning the eigenfunctions $G_\nu$:

\begin{enumerate}

  \item[(G1)] For all $\lambda > 0$ the limits
  \[
    G_{\lambda}^{+} = \lim_{\varepsilon \searrow 0} G_{\lambda+ i\varepsilon} \quad \text{and} \quad G_{\lambda}^{-}= \lim_{\varepsilon \searrow 0} G_{\lambda- i\varepsilon}
  \]
  exist in the usual $C^1$-topology (uniform convergence of the functions and their derivatives on compact sets of $(0,\infty)$; note that, using $D$, this implies convergence in the $C^2$-topology, and indeed in the $C^\infty$-topology).  Moreover the convergence is uniform over compact sets of positive $\lambda$.

  \item[(G2)] For all $\lambda > 0$ the functions $G^+_\lambda$ and $G^-_\lambda$ are linearly independent.
\end{enumerate}
The limit functions $G_\lambda ^\pm$ obey the relation
\begin{equation}
\label{eq-g-plus-minus-relation}
  \Wr(G^+_\lambda, G^- _\lambda) \cdot F_\lambda^{\phantom{+}} = G_\lambda^+ - G_\lambda^-
\end{equation}
for $\lambda >0$.  Indeed $G_\lambda^+$ and $G_\lambda ^-$ are $\lambda$-eigenfunctions for the differential operator $D$, and   as a result, the Wronskian $\Wr(G^+_\lambda, G^-_\lambda)$ is a constant function, so if we write
\[
  H_\lambda ^{\phantom{+}} = \Wr(G^+_\lambda, G^- _\lambda) \cdot F_\lambda^{\phantom{+}},
\]
then the three functions $ H_\lambda^{\phantom{+}}$, $G_\lambda^+$ and $G_\lambda^-$ all belong to the two-dimensional space of $\lambda$-eigenfunctions for $D$.  To verify that $H_\lambda^{\phantom{+}}= G^+_\lambda {-} G^-_\lambda$ we therefore just need to observe that
\[
  \Wr (H_\lambda^{\phantom{+}}, G_\lambda^{\pm}) = \Wr(G^+_\lambda {-} G^-_\lambda, G^\pm_\lambda),
\]
which is a consequence of \eqref{eq-wronskian-normalization1}. 
\begin{theorem}
  \label{thm-kodaira-2}
  If $ 0{<}\alpha{<}\beta$, then under the assumptions \textup{(G1)} and \textup{(G2)} above, the spectral projection $P_{(\alpha,\beta)}$ for $D$ is given by the formula
  \[
    (P_{(\alpha,\beta)} h )(x) = \int_0^\infty p_{(\alpha,\beta)}(x,y) h(y)\, dy ,
  \]
  for all smooth and compactly supported functions $h$ on $(0,\infty)$, where
  \begin{equation*}
    p_{(\alpha,\beta)}(x,y)
    = \frac{1}{2\pi i}  \int_\alpha^\beta 
    F_\lambda(x)F_\lambda (y)  
    \Wr (G_\lambda ^+, G_\lambda ^-) \,   {d\lambda} .
  \end{equation*}
\end{theorem}

\begin{proof}
It follows from Lemma~\ref{lem-kodaira-lemma},\eqref{eq-greens-function-fmla2} that 
\[
( P_{(\alpha,\beta)} h )(x) = \lim_{\varepsilon\searrow 0}\frac {1}{2 \pi i} 
 \int _\alpha^\beta  \Bigl ( \int_0^\infty \bigl( k_{\lambda+i\varepsilon}(x,y) {-} k_{\lambda -i\varepsilon} (x,y)\bigr ) h(y)\, dy \Bigr ) \, d\lambda 
\]
and from  \eqref{eq-def-of-k-nu}, together with the assumption (G1)  and  \eqref{eq-g-plus-minus-relation}, that 
\[
\lim_{\varepsilon\searrow 0} \bigl( k_{\lambda+i\varepsilon}(x,y) {-} k_{\lambda -i\varepsilon} (x,y)\bigr )   =
\Wr(G^{+}_\lambda, G^-_{\lambda})F_\lambda(x)F_\lambda (y).
\]
The convergence is uniform over compact subsets of $y\in (0,\infty)$ and compact subsets of $\lambda \in (0, \infty)$.  Hence
\[
( P_{(\alpha,\beta)} h )(x) = \frac {1}{2 \pi i}  \int_0^\infty  \Bigl(  \int _\alpha^\beta
\Wr(G^{+}_\lambda, G^-_{\lambda})F_\lambda(x)F_\lambda (y) \,d\lambda\Bigr ) \, h(y)\, dy
\]
as required.
\end{proof}

At this point we finally turn to Sturm-Liouville operators with eventually constant  coefficient functions, as in Section~\ref{sec-introduction}. If  we write 
\[
F_\nu (x) = c(\nu) \exp({i \sqrt{\nu}x}) + c(-\nu) \exp({-i \sqrt{\nu}x}) \qquad (x \gg 0)
\]
using the usual principal branch of the square root function, equal to the positive square root on the positive axis (we shall avoid the eigenvalues $\nu <0$), then using \eqref{eq-wronskian-normalization1} we compute that for $\lambda >0$ and $\nu  = \lambda {\pm} i\varepsilon$, 
\begin{equation}
\label{eq-fmla-for-g-nu}
G_{\nu  } (x) = \frac{i}{2c(\mp \nu  )  \sqrt{\nu  } } \; \exp(\pm {i\sqrt{{\nu }}x}) \qquad (x \gg 0) .
\end{equation}
The sign in the exponential is needed to ensure that $G_\nu$ is an $L^2$-function at  infinity, which is of course necessary if $G_\nu$ is to agree with a function in $\dom(D)$ at infinity.  It follows easily from \eqref{eq-fmla-for-g-nu} that (G1) and (G2) are satisfied, that 
\begin{equation*}
G^\pm _\lambda (x)  = \frac{\pm i}{2 c(\mp\lambda   )  \sqrt{\lambda   } } \; \exp(\pm {i\sqrt{{\lambda }}x}) \qquad (x \gg 0) ,
\end{equation*}
and that
\[
\Wr (G^+_\lambda , G_\lambda ^-) = \frac{i} {2 | c(\lambda)|^2 \sqrt{\lambda} } .
\]
Therefore Theorem~\ref{thm-kodaira-2} gives
\begin{equation*}
  \label{eq-projection4}
    p_{(\alpha,\beta)}  (x,y)
   = \frac{1}{4 \pi } \int _{\alpha }^{\beta } F_\lambda (x) F_\lambda (y) \frac{1}{ | c(\lambda)|^2} \frac{d\lambda}{ \sqrt{\lambda}} .
\end{equation*}
This is a reformulation of Weyl's theorem, as stated in Section~1.

\bibliography{Refs} \bibliographystyle{alpha}

\noindent {\small Department of Mathematics, Penn State University, University Park, PA 16802.}

\smallskip

\noindent{\small Email: higson@math.psu.edu and qut101@psu.edu.}

\end{document}